\documentclass[11pt]{amsart}

\usepackage{amssymb}
\usepackage{amsmath}
\usepackage{enumerate}
\usepackage[latin1]{inputenc}

\usepackage[T1]{fontenc}
\usepackage[sc]{mathpazo}
\usepackage[pdftex]{graphicx, color}
\usepackage{xypic}

\definecolor{alert}{rgb}{0.8,0,0}

\newcommand{\s}{\mathbb{S}}

\renewcommand{\r}{\mathbb{R}}

\renewcommand{\c}{\mathbb{C}}
\newcommand{\p}{\mathbb{P}}
\newcommand{\n}{\mathbb{N}}

\DeclareMathOperator{\Index}{Index}
\DeclareMathOperator{\pIm}{Im}

\newcommand{\prodesc}[2]{\left\langle{#1},{#2}\right\rangle}

\newtheorem{theorem}{Theorem}
\newtheorem{proposition}{Proposition}
\newtheorem{corollary}{Corollary}
\newtheorem{lemma}{Lemma}

\theoremstyle{definition}

\theoremstyle{remark}
  \newtheorem{remark}{Remark}

\numberwithin{equation}{section}

\begin{document}

\title[Second variation of one-sided complete minimal surfaces]{Second variation of one-sided complete minimal surfaces}

\author{Francisco Urbano}
\address{Departamento de Geometr\'{\i}a  y Topolog\'{\i}a \\
Universidad de Granada \\
18071 Granada, SPAIN}
\email{furbano@ugr.es}

\thanks{Research partially supported by a MCyT-Feder research project MTM2007-61775 and the Junta Andalucía Grants P06-FQM-01642.}



\date{}
\begin{abstract}
The stability and the index of complete one-sided minimal surfaces of certain $3$-dimensional Riemannian manifolds with positive scalar curvature are studied.
\end{abstract}

\maketitle

\section{Introduction}
The study of the second variation of the volume of minimal submanifolds into Riemannian manifolds can be considered as a classical problem in differential geometry. In fact, the operator of the second variation (the Jacobi operator) carries the information about the stability properties of the submanifold when it is thought as a stationary point for the volume functional. An important particular case in this setting is when the submanifold is a hypersurface. In this case, the rank of the normal bundle is one, and we can consider two different situations: the hypersurface is two-sided, i.e., the normal bundle is trivial, or the hypersurface is one-sided, i.e., the normal bundle is non-trivial. In the first case, we can define a global unit normal vector field which trivializes the normal bundle. Then, the Jacobi operator, which acts on the sections of the normal bundle, becomes a Schrödinger operator acting on functions. This case, perhaps the easiest one, has been studied by many people (see \cite{DRR, DP, FC, FCS, LR, P, U} and references therein). For one-sided minimal hypersurfaces and for minimal submanifolds with high codimension, only a few particular situations have been considered (see \cite{MW, MU, O, Oh, R, Rs, Si, TU} and references therein).

In the present work, we are interested in the second variation of complete minimal surfaces of Riemannian $3$-manifolds. Many interesting results are well-known for {\it compact} surfaces. So, when $M$ is the $3$-sphere $\s^3$, Simons \cite{Si} proved that the index of any compact mininal surface of $\s^3$ is at least one (in particular there are no stable ones) and the totally geodesic equators are the only ones with index one. Later, Urbano \cite{U} classified the orientable compact minimal surfaces of $\s^3$ with index less than six, proving that the surface must be either an equator or the Clifford torus which has index five.  When $M$ is the real projective space $\r\p^3$, Onhita \cite{O} proved that its only stable compact minimal surface is the totally geodesic real projective plane, and later, Do Carmo, Ritoré and Ros \cite{DRR} characterized the totally geodesic two-sphere and the Clifford torus as the only orientable two-sided compact minimal surfaces of $\r\p^3$ with index one.

For {\it complete two-sided} minimal surfaces in $3$-manifolds, the starting point was the characterization of the plane as the only complete stable two-sided minimal surface in the Euclidean space $\r^3$ (see \cite{FCS, DP, P}). Later, López-Ros \cite{LR} proved that the catenoid and the Enneper's surface are the only two-sided complete minimal surfaces of $\r^3$ with index one.  Fischer-Colbrie and Schoen \cite{FCS,FC} studied in depth the second variation of complete two-sided minimal surfaces of $3$-dimensional Riemannian manifolds with non-negative scalar curvature, analyzing, on any complete Riemannian surface, the index of the Schorödinger operator $L=\Delta-K+q$, where $\Delta$ is the Laplacian, $K$ is the Gauss curvature of the surface and $q$ is a non-negative function. In fact, the Jacobi operator of such minimal surfaces can be written in the above form. Also, in this setting, it is interesting to remark \cite{FCS, LR} that  a complete two-sided minimal surface with finite index of a $3$-dimensional Riemannian manifold with scalar curvature $\rho\geq\delta>0$ must be compact.

 For {\it complete one-sided} minimal surfaces in $3$-manifolds, the above Fischer-Colbrie and Schoen theory cannot be applied, and only some particular situations have been studied. Perhaps, the most interesting paper in this direction is the Ros' one \cite{R}, where he uses new ideas and proves, among other things, that there are no complete stable one-sided minimal surfaces in $\r^3$. A partial result of this was proved by Ross in \cite{Rs}.

In this paper, we study the stability and the index of complete minimal surfaces of $\s^3$, $\s^2\times\r$ and some of their Riemannian quotients like $\r\p^3$, $\r\p^2\times\r$, $\s^2\times\s^1$ and $\r\p^2\times\s^1$. The main results in the paper can be summarized in the following ones:
\begin{quote}

{\it The index of any complete and non-compact minimal surface of $\s^3$, $\r\p^3$, $\s^2\times\r$ or $\s^2\times\s^1$ is infinite.}
\end{quote}
\begin{quote}
{\it The totally geodesic embedding $\r\p^1\times\r\subset\r\p^2\times\r$ is the only stable orientable complete and non-compact minimal surface of $\r\p^2\times\r$.}
\end{quote}
\begin{quote}
{\it The totally geodesic embedding $\s^1\times\s^1\subset\s^2\times\s^1$ is the only compact orientable minimal surface of $\s^2\times\s^1$ with index one.}
\end{quote}
We remark that, when the surface is two-sided, the first result comes from Fischer-Colbrie and Schoen theory. Also, $\r\p^1\times\r$ is an orientable {\it one-sided} minimal surface in $\r\p^2\times\r$.

The author would like to thank A. Ros for his valuable comments about the paper.

\section{Harmonic vector fields on surfaces}
As in the proofs of some results we will use harmonic vector fields as test functions, in this section we recall some properties about harmonic vector fields, which will be used along the paper.

Given an orientable Riemannian surface $\Sigma$, a vector field $X$ on $\Sigma$ is {\it harmonic} if the associated $1$-form $\omega_X$ is harmonic, i.e., $\omega_X$ is closed and coclosed. This means that $\hbox{div}\,(X)=0$ and $\nabla X$ is a symmetric tensor, where $\hbox{div}$ is the divergence on $\Sigma$ and $\nabla$ is the Levi-Civita connection on $\Sigma$.

If $\Delta^{\Sigma}$ is the Laplacian on $\Sigma$ acting on vector fields and $X$ is a harmonic vector field, it is easy to check that
\begin{equation}\label{eq:harmonic}
\Delta^{\Sigma}X=KX,
\end{equation}
where $K$ is the Gauss curvature of $\Sigma$. Also, if $J$ is the complex structure on the Riemann surface $\Sigma$, then $X$ is harmonic if and only if $JX$ is harmonic. Let $H(\Sigma)$ be the space of square integrable harmonic vector fields on $\Sigma$. Then, if $\Sigma$ is compact of genus $g$, we have that $\hbox{dim}\,H(\Sigma)=2g$. If $\Sigma$ is non-compact, then $\hbox{dim}\,H(\Sigma)\geq 2\hbox{genus}\,(\Sigma)$, including the case where the genus of $\Sigma$ is infinite. (see \cite{FK}, pag.42).

If $\Sigma$ is a non-orientable Riemannian surface and $(\tilde{\Sigma},\tau)$ is its two-fold oriented covering, where $\tau$ is the change of sheet, then any harmonic vector field $X$ on $\tilde{\Sigma}$ decomposes as $X=X^++X^-$, where $X^+$ and $X^-$ are harmonic vector fields satisfying $\tau_*X^+=X^+$, $\tau_*X^-=-X^-$ and $JX^-=X^+$. In this case,
$H(\tilde{\Sigma})=H^+(\tilde{\Sigma})\oplus H^-(\tilde{\Sigma})$, where $H^{\pm}(\tilde{\Sigma})=\{X\in H(\tilde{\Sigma})\,|\,\tau_*X=\pm X\}$ and $J:H^+(\tilde{\Sigma})\rightarrow H^-(\tilde{\Sigma})$ is an isomorphism.

 It is interesting to remark an easy property which will be used in the paper. {\it Given a $k$-dimensional subspace $V$ of $H(\Sigma)$, there exists an integrable function $h$ on $\Sigma$ such that $|X|^2\leq h$ for any $X\in V$ with  $\int_\Sigma |X|^2=1$}. In fact if $\{V_1,\dots,V_k\}$ is a $L^2$-orthonormal basis of $V$, then $X=\sum_{i=1}^k\lambda_iV_i$ with $\sum\lambda_i^2=1$. Now it is clear that $h$ can be taken as the integrable function $k^2\max\{\langle V_i,V_j\rangle,\,1\leq i,j\leq k\}$.

\section{Jacobi operator}

Let $\Phi:\Sigma\rightarrow (M^3,\langle,\rangle)$ be a minimal immersion of a surface $\Sigma$ in a $3$-dimensional Riemannian manifold $M$. The Jacobi operator of the second variation of the area is a strongly elliptic operator acting on sections of the normal bundle, $L:\Gamma(T^{\bot}\Sigma)\rightarrow \Gamma(T^{\bot}\Sigma)$, given by
\[
L=\Delta^{\bot}+|\sigma|^2+Ric(n),
\]
where $\Delta^{\bot}$ is the normal Laplacian, $\sigma$ is the second fundamental form of $\Phi$ and Ric(n) is the Ricci curvature of any unit normal vector $n$.

If $\Omega$ is a compact domain of $\Sigma$, the operator $L$, with zero boundary conditions, has a discrete spectrum $\lambda_1(\Omega)<\lambda_2(\Omega)<\dots\rightarrow\infty$, and the dimension of each eigenspace is finite. The index of $L$ in $\Omega$, $\Index(L,\Omega)$, is the sum of the dimensions of the eigenspaces corresponding to negative eigenvalues.

The {\it index of the minimal} immersion $\Phi:\Sigma\rightarrow M$ is the index of the operator $L$ on $\Sigma$, which is defined by
\[
\Index(\Phi)=\Index(L):=\sup\{\Index(L,\Omega)\,|\,\Omega\;\hbox{compact domain of}\;\Sigma\}.
\]
The minimal immersion $\Phi$ is called {\it stable} if $\Index(\Phi)=0$. This means that the quadratic form  $Q:\Gamma_0(T^{\perp}\Sigma)\rightarrow\r$ associated to $L$
$$
Q(\eta)=-\int_{\Sigma}\langle L\eta,\eta\rangle\,dv=\int_{\Sigma}\{|\nabla^{\bot}\eta|^2-(|\sigma|^2
+Ric(n))|\eta|^2\}dA\geq 0 ,
$$
for any compactly supported normal section $\eta\in\Gamma_0(T^{\perp}\Sigma)$.

The immersion $\Phi$ is called {\it two-sided} if $T^{\bot}\Sigma$ is trivial, i.e., there exists a global unit normal vector field $N$. Otherwise, i.e., when $T^{\bot}\Sigma$ is non-trivial, the immersion is called {\it one-sided}. When the ambient manifold $M$ is orientable, $\Phi$ is two-sided if and only if $\Sigma$ is orientable. This property is not true when $M$ is non-orientable.

If $\Phi$ is two-sided, sections of the normal bundle can be identified with functions on the surface in the following way:
\begin{eqnarray*}
\Gamma(T^{\perp}\Sigma)&\equiv& C^{\infty} (\Sigma)\\
\eta&\equiv& f,\quad\quad\hbox{if}\,\,  \eta=fN,
\end{eqnarray*}
where $N$ is a global unit normal section to $\Phi$. In this case, it is clear that $\Delta^{\perp}\eta=(\Delta f)N$, and hence the Jacobi operator becomes a Schr\"odinger operator acting on functions
$L:C^{\infty}(\Sigma)\rightarrow C^{\infty}(\Sigma)$, given by
\[
L=\Delta+|\sigma|^2+Ric(N)=\Delta-K+(|\sigma|^2+\rho)/2,
\]
where $K$ is the Gauss curvature of $\Sigma$, $\rho$ scalar curvature of $M$ and we have used the Gauss equation of $\Phi$ to obtain the second expression of $L$.

When $M$ is the $3$-dimensional unit sphere $\s^3$ or the $3$-dimensional real projective space $\r\p^3$, the Jacobi operator is given by
\[
L=\Delta^{\bot}+|\sigma|^2+2,
\]
whereas if $M$ is the Riemannian product $\s^2\times\r$, $\r\p^2\times\r$, or their quotients $\s^2\times\s^1(r)$, $\r\p^2\times\s^1(r)$, the Jacobi operator is
\[
L=\Delta^{\bot}+|\sigma|^2+|\xi^{\top}|^2,
\]
where $\s^1(r)$ is the circle of radius $r$, $\xi$ is a unit parallel vertical vector field on the ambient manifold and $\top$ stands for the tangential component.

Now, we will expose some background about minimal surfaces, which will be used later. Let $\Phi=(\phi,\psi):\Sigma\rightarrow \s^2\times\r$ (respectively in $\s^2\times\s^1(r), \r\p^2\times\r,\r\p^2\times\s^1(r)$) be a minimal immersion of a surface $\Sigma$ and denote also by $\langle,\rangle$ the induced metric. If $\bar{R}$ denotes the curvature operator of the ambient $3$-manifold, it is easy to prove that $\bar{R}(e_1,e_2,e_2,e_1)=1-|\xi^{\top}|^2$, where $\{e_1,e_2\}$ is an orthonormal basis on $\Sigma$. So, the Gauss equation of $\Phi$ can be written as
\[
K=1-|\xi^{\top}|^2-\frac{|\sigma|^2}{2}.
\]

If $\Sigma$ is orientable and  $z=x+iy$ is a conformal parameter with induced metric $e^{2u}|dz|^2$ and $\partial_z=(\partial_x-i\partial_y)/2$, $\partial_{\bar{z}}=(\partial_x+i\partial_y)/2$ are the corresponding complex operators, then it is well-known that
\[
\Theta(z)=\langle\Phi_z,\xi\rangle\,dz
\]
is a globally defined holomorphic $1$-differential on $\Sigma$. As $|\Theta|^2=e^{2u}|\xi^{\top}|^2/2$, we have that either $\xi^{\top}=0$ (i.e. $\Theta\equiv 0$) or $\{p\in\Sigma\,|\,\xi^{\top}(p)=0\}$ is isolated. In the first case, $\xi$ is normal to $\Phi$, and then $d\psi(v)=0$ for any tangent vector $v$, i.e. $\psi$ is constant. Moreover $\phi:\Sigma\rightarrow \s^2$ (respectively $\phi:\Sigma\rightarrow \r\p^2$) is a local isometry. In this case we will say that $\Sigma$ is a slice. The Gauss equation says that the slices are totally geodesic surfaces.

\begin{lemma}\label{lm:Jacobi-harmonic}
Let $\Phi:\Sigma\rightarrow (M^3,\langle,\rangle)$ be a minimal immersion of an orientable surface $\Sigma$ and $X$  any harmonic vector field on $\Sigma$.
\begin{enumerate}
\item If $M=\s^3$, then $LX=2X+2\langle\sigma,\nabla X\rangle N,$
\item If $M=\s^2\times\r$, then
\[
\langle LX,X\rangle=(2-|\xi^{\top}|^2)\langle X,\xi\rangle^2,
\]
\item If $M=\s^2\times\s^1(r)$, then
\[
\langle LX,X\rangle=\left(2-(1+\frac{1}{r^2})|\xi^{\top}|^2\right)\langle X,\xi\rangle^2,
\]
\end{enumerate}
where $L$ is the Jacobi operator of the two-sided minimal immersion $\Phi$ and $X$ is being considered as a $\r^4$-valuated function in (1) and (2) and as a $\r^5$-valuated function in (3).
\end{lemma}
\begin{proof}
We consider $\s^3, \s^2\times\r\subset\r^4$ and $\s^2\times\s^1(r)\subset\r^5$. If $\nabla^0$ is the connection on $\r^4$ or $\r^5$ and $\bar{\sigma}$ the second fundamental form of $M$ in $\r^4$ or $\r^5$, then
\[
\nabla^0_vX=\nabla_vX+\sigma(v,X)+\bar{\sigma}(v,X),
\]
and so, using ~\eqref{eq:harmonic}
\[
\Delta^0X=(K-|\sigma|^2/2)X+2\langle\sigma,\nabla X\rangle N+\sum_{i=1}^2\{-\bar{A}_{\bar{\sigma}(e_i,X)}e_i+2\bar{\sigma}(e_i,\nabla_{e_i}X+\sigma(e_i,X))\}
\]
where $\{e_1,e_2\}$ is an orthonormal reference on $\Sigma$ and $\bar{A}$ is the Weingarten endomorphism of $M$ in $\r^4$ or $\r^5$. Hence,
\[
LX=(\rho/2)X+2\langle\sigma,\nabla X\rangle N+\sum_{i=1}^2\{-\bar{A}_{\bar{\sigma}(e_i,X)}e_i+2\bar{\sigma}(e_i,\nabla_{e_i}X+\sigma(e_i,X))\}.
\]
Now, using the expressions of the second fundamental forms of these three manifolds in $\r^4$ and $\r^5$, it is easy to prove the Lemma.
\end{proof}
\begin{lemma}\label{lm:campos-armonicos}
Let $\Phi:\Sigma\rightarrow \s^2\times\r$ (respectively $\s^2\times\s^1(r)$) be a minimal immersion of an orientable surface $\Sigma$. If $\Phi$ is not a slice, then
\begin{enumerate}
\item The tangential component $\xi^{\top}$ of the vertical vector field $\xi$ is a harmonic vector field on $\Sigma$ with only a discrete number of zeroes. Moreover, the harmonic vector field $J\xi^{\top}$ satisfies $\langle J\xi^{\top},\xi\rangle=0$ and any vector field $X$ on $\Sigma$ can be written, almost everywhere, as $X=f \xi^{\top}+gJ\xi^{\top}$, for certain functions $f$ and $g$.
\item If $(\Sigma,\tau)\rightarrow \Sigma_0$ is the two-fold oriented covering of a non-orientable surface $\Sigma_0$ and $\Phi$ is the lift of a minimal immersion $\Phi_0:\Sigma_0\rightarrow\s^2\times\r$ (respectively $\s^2\times\s^1(r)$), then $\tau_*\xi^{\top}=\xi^{\top}$ and $\tau_*J\xi^{\top}=-J\xi^{\top}$.
\end{enumerate}
\end{lemma}
\begin{proof}
Using that $\xi$ is a parallel vector field, it is clear that for any $v,w$ tangent to $\Sigma$ we have that
\[
\langle\nabla_v\xi^{\top},w\rangle=\langle\sigma(v,w),\xi\rangle.
\]
This means that $\xi^{\top}$ is a harmonic vector field on $\Sigma$. Also, the harmonic vector field $J\xi^{\top}$ is perpendicular to $\xi$, and so $\{(\xi^{\top})_p,(J\xi^{\top})_p\}$ are linearly independent on $\{p\in\Sigma\,|\,\xi^{\top}(p)\not= 0\}$. This proves (1).

If $\Pi:\Sigma\rightarrow\Sigma_0$ is the projection, then $\Phi=\Phi_0\circ\Pi$ and so $\Phi\circ\tau=\Phi$. This implies that $\tau_*\xi^{\top}=\xi^{\top}$. Also, as $\tau_*\circ J=-J\circ\tau_*$, we have that $\tau _*J\xi^{\top}=-J\xi^{\top}$.

\end{proof}
\section{Statement and proof of the main results}
In this section we will use some results which appear explicitly in \cite{FCS, FC, LR} or follow from them. For completeness we will next describe them.
The next result can be proved following the same arguments like in Proposition 2 of \cite{FC}, and so we omit the proof.
\begin{proposition}[\cite{FC}]
Let $\Sigma$ be an orientable complete Riemannian surface and $\tau$ an isometry of $\Sigma$ without fixed points and with $\tau^2=Id$. Let $L=\Delta +q$ be a Schrödinger operator on $\Sigma$ with $q\circ\tau=q$ and consider the operator $$L^-=L_{| C^{\infty}_-(\Sigma)}:C^{\infty}_-(\Sigma)\rightarrow C^{\infty}_-(\Sigma),$$ where $C^{\infty}_-(\Sigma)=\{f\in C^{\infty}(\Sigma)\,|\,f\circ\tau=-f\}$. Then $L^-$ has finite index $k$ if and only if there exists a $k$-dimensional subspace $W$ of $L^2_-(\Sigma)$ having an orthonormal basis $\{v_1,\dots,v_k\}$ with $Lv_i+\lambda_iv_i=0,\,\lambda_i<0$ and $Q(f)\geq 0$ for any function $f\in C^{\infty}_0(\Sigma)\cap W^{\bot}$ with $f\circ\tau=-f$.
\end{proposition}

\begin{theorem}[\cite{FCS, FC, LR}]
Let $\Sigma$ be a complete Riemannian surface, $L=\Delta-K+q$ a Schr\"odinger operator on $\Sigma$, where $K$ is the Gauss curvature of $\Sigma$ and $q\geq 0$.
\begin{enumerate}
\item If $\Sigma$ is orientable and $\Index(L)=0$, then: either $\Sigma$ is conformally equivalent to the sphere $\s^2$ or the complex plane $\c$, or $q=0$ and $\Sigma$ is either a a flat torus  or flat cylinder.
\item If $q\geq c>0$ for some constant $c$, and there exists a compact set $C\subset\Sigma$ such that $\Index(L)=0$ on $\Sigma-C$, then $\Sigma$ is compact.
\end{enumerate}
\end{theorem}
From here, we obtain that
\begin{corollary}[\cite{FC, LR}]
{\it If $\Phi:\Sigma\rightarrow M^3$ is a two-sided minimal immersion of a complete and non-compact surface in a Riemannian manifold $M$ with scalar curvature $\rho\geq c>0$, then $\Index\,(\Phi)=\infty$. }
\end{corollary}
\begin{proof}
As $\Phi$ is two-sided, the Jacobi operator is the Schrödinger operator $L=\Delta-K+(\rho+|\sigma|^2)/2$. If $\Index(\Phi)<\infty$, from Proposition 1 in \cite{FC}, there exists a compact set $C\subset\Sigma$ such that $\Sigma-C$ is stable. Now the result follows from Theorem 1,(2).
\end{proof}
Using similar ideas, we can extend Corollary 1 to a certain family of one-sided complete minimal surfaces.
\begin{corollary}\label{cor:finite-genus}
Let $\Phi:\Sigma\rightarrow M^3$ be a one-sided minimal immersion of a complete and non-compact surface $\Sigma$ in an orientable Riemannian manifold $M$ with scalar curvature $\rho\geq c>0$. If the genus of the $2$-fold oriented covering of $\Sigma$ is finite, then $\Index(\Phi)=\infty$.
\end{corollary}
\begin{remark}
As $\Phi$ is one-sided and $M$ is orientable, $\Sigma$ is not orientable. Also, the orientability of the ambient manifold is necessary in the assumptions, because $\r\p^1\times\r$ is an orientable complete surface of genus zero, which is embedded in $\r\p^2\times\r$ as a stable minimal one-sided surface (see Theorem~\ref{thm:estable}).
\end{remark}
\begin{proof}
Let $(\tilde{\Sigma},\tau)$ be the $2$-fold oriented covering of $\Sigma$ with $\tau$ the change of sheet on $\tilde{\Sigma}$. As the genus of $\tilde{\Sigma}$ is finite, Lemma 9 in \cite{R} implies that there exists a compact subset $C\subset\Sigma$, such that $\Sigma-C$ is orientable. Hence, as $M$ is orientable, $\Phi:\Sigma-C\rightarrow M$ is a two-sided minimal immersion and so the Jacobi operator on $\Sigma-C$ is $L=\Delta-K+(\rho+|\sigma|^2)/2$.

If $ \Index (\Sigma)$ is finite, then $\Index(\Sigma-C)$ is finite too. So, from Proposition 1 in \cite{FC}, there exists a compact subset $K\subset \Sigma-C$, such that $\Sigma-(C\cup K)$ is stable. Hence $\Sigma$ is a complete surface and the Schr\"odinger operator $L=\Delta-K+(\rho+|\sigma|^2)/2$ on $\Sigma$ satisfies that $\Index (L)=0$ on $\Sigma-(C\cup K)$. Theorem 1,(2) says that $\Sigma$ must be compact. This proves the Corollary.
\end{proof}
Hence, for orientable ambient $3$-manifolds with scalar curvature $\rho\geq c>0$, the remaining case to study is when the two-fold oriented covering of the one-sided complete and non-compact minimal surface has infinite genus. We have not obtained a general result, but following \cite{R} and using harmonic vector fields as test functions, we have finished the study in some particular three-manifolds.
\begin{theorem}\label{thm:indice-infinito}
 Let $\Phi:\Sigma\rightarrow M$ be a minimal immersion of a complete surface $\Sigma$ in a $3$-dimensional Riemannian manifold $M$.
\begin{enumerate}
\item If $M=\s^3$ and $\Sigma$ is not compact, then $\Index (\Phi)=\infty$.
\item If $M=\s^2\times\r$, then either $\Sigma=\s^2$, $\Phi(\Sigma)=\s^2\times\{t\},\,t\in\r$ and $\Phi$ is stable, or
$\Index(\Phi)=\infty$.
\item If $M=\s^2\times\s^1(r),\,r\geq 1$ and $\Sigma$ is not compact, then $\Index (\Phi)=\infty$.
\end{enumerate}
\end{theorem}
\begin{proof}
As the only compact minimal surfaces of $\s^2\times\r$ are the slices $\s^2\times\{t\},\,t\in\r$, which are stable, then Corollaries 1 and 2 imply that, we can assume that $\Phi$ is one-sided, i.e., $\Sigma$ is not orientable and that the two-fold oriented covering $\tilde{\Sigma}$ of $\Sigma$ has infinite genus.

Let $\tau$ be the change of sheet in $\tilde{\Sigma}$ and $\Pi:\tilde{\Sigma}\rightarrow\Sigma$ the projection. Then, $\tilde{\Phi}=\Phi\circ\Pi$ is also a minimal immersion which is two-sided. Let $\tilde{N}$ be a global unit normal vector field to $\tilde{\Phi}$. As $\Phi$ is one-sided, $\tilde{N}\circ\tau=-\tilde{N}$. Also, we can identify sections of the normal bundle of $\Phi$, $\Gamma(T^{\bot}\Sigma)$, with functions on $\tilde{\Sigma}$ which are odd with respect to $\tau$:
\begin{eqnarray}
\Gamma(T^{\bot}\Sigma)&\equiv& C^{\infty}_-(\tilde{\Sigma})=\{f\in C^{\infty}(\tilde{\Sigma})\,|\,f\circ\tau=-f\}\\
\eta&\equiv& f,\quad\quad \hbox{if}\,\,\tilde{\eta}=f\tilde{N},
\end{eqnarray}
where $\tilde{\eta}$ is the lift of $\eta$ to $\tilde{\Phi}$ . Also, if $\tilde{L}=\tilde{\Delta}-\tilde{K}+\rho/2+|\tilde{\sigma}|^2/2$ is the Jacobi operator of $\tilde{\Phi}$, it is clear that
\[
2Q(\eta)=\tilde{Q}(\tilde{\eta})=\tilde{Q}(f)=-\int_{\tilde{\Sigma}}f\tilde{L}f\,d\tilde{A},
\]
for any compactly supported $\eta\in\Gamma(T^{\bot}\Sigma)$.

To prove the result, we suppose that $\Index(\Phi)=k$ and we will find a contradiction.

From Proposition 1, there exist $L^2(\tilde{\Sigma})$-functions $\{v_1,\dots,v_k\}$ with $v_i\circ\tau=-v_i$ and $\tilde{L}v_i+\lambda_iv_i=0,\,\lambda_i<0$, and such that, if $f\in C_0^{\infty}(\tilde{\Sigma})$ satisfies $f\circ\tau=-f$ and is $L^2$-orthogonal to $v_i,\,1\leq i\leq k$, then
\[
\tilde{Q}(f)\geq 0.
\]
As the genus of $\tilde{\Sigma}$ is infinite, the space $H(\tilde{\Sigma})$ of $L^2$-harmonic vector fields on $\tilde{\Sigma}$ has also infinite dimension. So, let $V$ be a $l$-dimensional ($l>4k$ if $M=\s^3$ or $\s^2\times\r$ and $l>5k$ if $M=\s^2\times\s^1(r)$) subspace of $H^-(\tilde{\Sigma})$. When $M=\s^2\times\r\quad\hbox{or}\quad\s^2\times\s^1(r)$, we will take $V$ with another restriction. In fact, in these cases, if the harmonic vector field $J\xi^{\top}$ (see Lemma~\ref{lm:campos-armonicos}) satisfies $\int_{\tilde{\Sigma}}|J\xi^{\top}|^2<\infty$, i.e., $J\xi^{\top}\in H^-(\tilde{\Sigma})$, we will take $V$ being $L_2$-orthogonal to the line spanned by $J\xi^{\top}$. If $\int_{\tilde{\Sigma}}|J\xi^{\top}|^2=\infty$, we will not impose more restrictions to $V$.

 Let $\{U_n\,|\,n\in\n\}$ be an exhaustion of the complete surface $\tilde{\Sigma}$ and $\{\varphi_n\,|\, n\in \n\}$ cut-off functions on $\tilde{\Sigma}$ with $\hbox{supp}\,(\varphi_n)\subset U_n$, $|\nabla\varphi_n|^2\leq 1$ and $\varphi_n\circ\tau=\varphi_n$.

For each $n\in\n$, let $F_n:V\rightarrow \r^{4k}$ (respectively $\r^{5k}$) be the linear map given by
\[
F_n(X)=\left(\int_{\tilde{\Sigma}}\varphi_nv_1X,\dots,\int_{\tilde{\Sigma}}\varphi_nv_kX\right).
\]
As $l>4k$ (respectively $l>5k$), let $X_n\in \ker F_n$ with $\int_{\tilde{\Sigma}}|X_n|^2=1$. So, for each $n$, the function $\varphi_nX_n$ has compact support, is $L^2$-orthogonal to $v_i,\,1\leq i\leq  k$ and, as $X_n\in H^-(\tilde{\Sigma})$, we have that $\varphi_nX_n\circ\tau=-\varphi_nX_n$. Hence
\[
\tilde{Q}(\varphi_nX_n)\geq 0,\quad \forall n\in\n.
\]
Now,
\begin{eqnarray*}
\tilde{Q}(\varphi_nX_n)=-\int_{\tilde{\Sigma}}\varphi_n\tilde{\Delta}\varphi_n|X_n|^2-\int_{\tilde{\Sigma}}\frac{1}{2}\langle\tilde{\nabla}\varphi_n^2,\tilde{\nabla}|X_n|^2\rangle
-\int_{\tilde{\Sigma}}\varphi_n^2\langle \tilde{L}X_n,X_n\rangle.
\end{eqnarray*}
As $\hbox{div}\,(|X_n|^2\tilde{\nabla}\varphi_n^2)=\langle\tilde{\nabla}\varphi_n^2,\tilde{\nabla}|X_n|^2\rangle+2(\varphi_n\tilde{\Delta}\varphi_n+|\tilde{\nabla}\varphi_n|^2)|X_n|^2$, we finally obtain that
\[
\int_{\tilde{\Sigma}}\varphi_n^2\langle \tilde{L}X_n,X_n\rangle\leq\int_{\tilde{\Sigma}}|\tilde{\nabla}\varphi_n|^2|X_n|^2,\quad \forall n\in\n.
\]
Using Lemma~\ref{lm:Jacobi-harmonic}, we obtain that
\[
\int_{\tilde{\Sigma}}|\tilde{\nabla}\varphi_n|^2|X_n|^2\geq
\begin{cases}
 2\int_{\tilde{\Sigma}}\varphi_n^2|X_n|^2,\quad\hbox{when}\quad M=\s^3,\\
 \int_{\tilde{\Sigma}}\varphi_n^2(2-|\xi^{\top}|^2)\langle X_n,\xi\rangle^2,\quad\hbox{when}\quad M=\s^2\times\r,\\
 \int_{\tilde{\Sigma}}\varphi_n^2(2-(1+\frac{1}{r^2})|\xi^{\top}|^2)\langle X_n,\xi\rangle^2,\quad\hbox{when}\quad M=\s^2\times\s^1(r).
\end{cases}
\]
It is clear that the sequence of harmonic vector fields $X_n$ converges, up to extractions a subsequence, in $V$ to a harmonic vector field $X$ with $\int_{\tilde{\Sigma}}|X|^2=1$. Using the remark made at the end of Section 2, we have that
\[
 \lim\int_{\tilde{\Sigma}}|\nabla\varphi_n|^2|X_n|^2=\int_{\tilde{\Sigma}}\lim|\nabla\varphi_n|^2|X_n|^2=0,
 \]
because $|\nabla\varphi_n|^2\rightarrow 0.$

So, when $M=\s^3$,
\[
0=\lim\int_{\tilde{\Sigma}}\varphi_n^2|X_n|^2=\int_{\tilde{\Sigma}}\lim\varphi_n^2|X_n|^2=\int_{\tilde{\Sigma}}|X|^2=1,
\]
which is a contradiction.
When $M=\s^2\times\r$, we have that
\begin{eqnarray*}
0&=& \lim\int_{\tilde{\Sigma}}\varphi_n^2(2-|\xi^{\top}|^2)\langle X_n,\xi\rangle^2\\&=&\int_{\tilde{\Sigma}}\lim\varphi_n^2(2-|\xi^{\top}|^2)\langle X_n,\xi\rangle^2=\int_{\tilde{\Sigma}}(2-|\xi^{\top}|^2)\langle X,\xi\rangle^2,
\end{eqnarray*}
which implies that $\langle X,\xi\rangle=0$.

When $M=\s^2\times\s^1(r)$, we have that
\[
0= \lim\int_{\tilde{\Sigma}}\varphi_n^2(2-(1+\frac{1}{r^2})|\xi^{\top}|^2)\langle X_n,\xi\rangle^2=\int_{\tilde{\Sigma}}(2-(1+\frac{1}{r^2})|\xi^{\top}|^2)\langle X,\xi\rangle^2.
\]
As $r\geq 1$, we have that either $r = 1$ and $|\xi^{\top}|^2=1$ or $\prodesc{X}{\xi} = 0$. In the first case, $\tilde{\Sigma}$ is a covering of the totally geodesic torus $\s^1 \times \s^1$, and so it genus cannot be infinite. So, $\langle X,\xi\rangle=0$.

Hence in the last two cases and using Lemma~\ref{lm:campos-armonicos}, we obtain that the harmonic vector field $X\in V$ is given by $X=\lambda J\xi^{\top}$ for certain nonzero real number $\lambda$. Hence, $1=\lambda^2\int_{\tilde{\Sigma}}|J\xi^{\top}|^2$ and so $J\xi^{\top}\in H^-(\tilde{\Sigma})$. As, in this case, $J\xi^{\top}$ is orthogonal to $V$, we get again a contradiction. This finishes the proof.
\end{proof}
Now, we extend the results of Theorem~\ref{thm:indice-infinito} to minimal surfaces in the real projective space. Although we use a similar idea, the proof is more complicated and in it we will use different test functions.
\begin{theorem}\label{thm:proyectivo}
Let $\Phi:\Sigma\rightarrow \r\p^3$ be a minimal immersion of a complete non-compact surface. Then $\Index (\Phi)=\infty$.
\end{theorem}
\begin{proof}
Following a reasoning like in the previous theorem, we can assume that $\Phi$ is one-sided, i.e., $\Sigma$ is not orientable, and that the two-fold oriented covering $(\tilde{\Sigma},\tau)$ of $\Sigma$ has infinite genus.

Then, $\tilde{\Phi}=\Phi\circ\Pi$ is a two-sided minimal immersion of $\tilde{\Sigma}$ into $\r\p^3$. We consider two cases.

First case: $\tilde{\Phi}$ admits a lift $\Psi$ to $\s^3$.
\[
\xymatrix{
(\tilde{\Sigma}, \tau)  \ar[drr]^{\tilde{\Phi},\, 2\text{-sided}} \ar[d]_{\Pi} \ar[rr]^{\Psi}    &   & \s^3 \ar[d] \\
\Sigma  \ar[rr]^{\Phi}_{1\text{-sided}}                       &   & \r\p^3
}
\]
 Then, we can identify sections of the normal bundle to $\Phi$, $\Gamma(T^{\bot}\Sigma)$, with functions on $\tilde{\Sigma}$ which are odd with respect to $\tau$:
\begin{eqnarray*}
\Gamma(T^{\bot}\Sigma)\equiv C^{\infty}_-(\tilde{\Sigma})\\
\eta\equiv f,\quad\quad \hbox{if}\quad\tilde{\eta}=f\tilde{N},
\end{eqnarray*}
where $\tilde{\eta}$ is the lift of $\eta$ to $\Psi$. Moreover the quadratic forms associated to $\Phi$ and $\Psi$ satisfy
\[
2Q(\eta)=\tilde{Q}(\tilde{\eta})=\tilde{Q}(f),
\]
for any compactly supported $f\in C^{\infty}_-(\tilde{\Sigma})$. Now, the immersion $\Psi$ is under the same conditions that the immersion $\tilde{\Phi}$ in Theorem~\ref{thm:indice-infinito},(1). So the result follows making the same proof than in Theorem~\ref{thm:indice-infinito},(1).

Second case: $\tilde{\Phi}$ cannot be lifted to $\s^3$. Then, the two-fold covering $\s^3\rightarrow\r\p^3$ induces a two-fold covering $\hat{\Sigma}\rightarrow\tilde{\Sigma}$ of a connected surface $\hat{\Sigma}$. If $\hat{\tau}$ is the change of sheet, then $\hat{\Phi}\circ\hat{\tau}=-\hat{\Phi}$ and $\hat{\Phi}$ is two-sided.

\[
\xymatrix{
(\hat{\Sigma}, \hat{\tau}) \ar[d]_{\hat{\Pi}} \ar[drr]^{\hat{\Phi}}  &   &   \\
(\tilde{\Sigma}, \tau)  \ar[drr]^{\tilde{\Phi},\, 2\text{-sided}} \ar[d]_{\Pi}     &   & \s^3 \ar[d]\\
\Sigma  \ar[rr]^{\Phi}_{1\text{-sided}}                       &   & \r\p^3
}
\]
Now, we can lift $\tau$ to $\hat{\Sigma}$ as follows: given $x\in\hat{\Sigma}$, $\tau(x)=y$ if $\tau(\hat{\Pi}(x))=\hat{\Pi}(y)$ and $\hat{\Phi}(x)=\hat{\Phi}(y)$. As $y$ is uniquely determined, the lift of $\tau$ to $\hat{\Sigma}$, which will be denoted also by $\tau$, is well-defined.

Let $N$ be a unit normal vector field to $\hat{\Phi}$. As $\tilde{\Phi}$ is two-sided and $\Phi$ is one-sided, then $N\circ\tau=-N$. Also, as $N$ projects to $\tilde{\Phi}$, then $N\circ\hat{\tau}=-N$.

Now, we can identify normal sections of $\Phi$ with functions on $\hat{\Sigma}$ with the following symmetries:
\begin{eqnarray*}
\Gamma(T^{\perp}\Sigma)&\equiv& C^{\infty}_{\pm}(\hat{\Sigma})=\{f\in C^{\infty}(\hat{\Sigma})\,|\,f\circ\hat{\tau}=f,\,f\circ\tau=-f\} \\
\eta&\equiv& f,\quad \hat{\eta}=fN,
\end{eqnarray*}
where $\hat{\eta}$ is the lift of $\eta$ to $\hat{\Phi}$. It is clear that $\hat{\eta}\circ\tau=-\hat{\eta}\circ\hat{\tau}=\hat{\eta}$, and so the function $f$ satisfies $f\circ\hat{\tau}=-f\circ\tau=f$.

If $Q$ and $\hat{Q}$ are the quadratic forms associated to the minimal immersions $\Phi$ and $\hat{\Phi}$,
\[
4Q(\eta)=\hat{Q}(\hat{\eta})=\hat{Q}(f)=-\int_{\hat{\Sigma}}f\hat{L}f,
\]
for any compactly supported function $f\in C^{\infty}_{\pm}(\hat{\Sigma})$, where
\begin{eqnarray*}
\hat{L}:C^{\infty}_{\pm}(\hat{\Sigma})\rightarrow C^{\infty}_{\pm}(\hat{\Sigma})\\
\hat{L}=\hat{\Delta}-\hat{K}+(3+|\hat{\sigma}|^2/2).
\end{eqnarray*}
To prove the result, we suppose that $\Index(\Phi)=k$ and we will find a contradiction.

 From Proposition 1, there exist $L^2(\hat{\Sigma})$-functions $\{v_1,\dots,v_k\}$ with $v_i\circ\tau=-v_i,\,v_i\circ\hat{\tau}=v_i$ and $\hat{L}v_i+\lambda_iv_i=0,\,\lambda_i<0$, and such that if $f$ is a compactly supported function in $C^{\infty}_{\pm}(\hat{\Sigma})$ which is $L^2$-orthogonal to $v_i,\,1\leq i\leq k$, then
\[
\hat{Q}(f)\geq 0.
\]
As the genus of $\hat{\Sigma}$ is infinite, the space $H(\hat{\Sigma})$ of $L^2$-harmonic vector fields on $\hat{\Sigma}$ has also infinite dimension. So, let $V$ be a $l$-dimensional ($l>16k$) subspace of $H(\hat{\Sigma})$, such that $\tau_*X=-X$ and $\hat{\tau}_*X=X$ for any $X\in V$. Let $Y$ be any compactly supported function in $C^{\infty}_{\pm}(\hat{\Sigma},\r^4)$ which is $L^2$-orthogonal to $v_j,j=1,2,3,4$. Let $\{U_n\,|\,n\in\n\}$ be an exhaustion of the complete surface $\hat{\Sigma}$ and $\{\varphi_n\,|\, n\in \n\}$ cut-off functions on $\hat{\Sigma}$ with $\hbox{supp}\,(Y)\subset\overline{\{p\in\tilde{\Sigma}\,|\,\varphi_n(p)=1\}}$, $\hbox{supp}\,(\varphi_n)\subset U_n$, $|\nabla\varphi_n|^2\leq 1$ and $\varphi_n\circ\tau=\varphi_n\circ\hat{\tau}=\varphi_n$.

Let $\{a_1,\dots,a_4\}$ be an orthonornal basis of $\r^4$. For each $n\in\n$, let $F_n:V\rightarrow \r^{16k}$ be the linear map given by
\[
F_n(X)=\left(\int_{\hat{\Sigma}}\varphi_nv_i\langle\hat{\Phi},a_j\rangle X\right)_{ij}.
\]
As $l>16k$, there exists $X_n\in \ker F_n$ with $\int_{\hat{\Sigma}}|X_n|^2=1$. So, for each $n$ and each $j\in\{1,2,3,4\}$, the functions $\varphi_n\langle\hat{\Phi},a_j\rangle X_n$ have compact support and they are $L^2$-orthogonal to $v_i,\, 1\leq i\leq k$. Using the properties of $X_n$ and $\varphi_n$ with respect to $\tau$ and $\hat{\tau}$ and the fact that $\hat{\Phi}\circ\tau=\hat{\Phi}$ and $\hat{\Phi}\circ\hat{\tau}=-\hat{\Phi}$, we obtain that $\varphi_n\langle\hat{\Phi},a_j\rangle X_n\circ\tau=-\varphi_n\langle\hat{\Phi},a_j\rangle X_n\circ\hat{\tau}=-\varphi_n\langle\hat{\Phi},a_j\rangle X_n$.

Now, for each $t\in\r$, let $F_t=\varphi_n\langle\hat{\Phi},a_j\rangle X_n+tY$. Then $F_{t}\in C^{\infty}_{\pm}(\hat{\Sigma},\r^4)$, has compact support and is $L^2$-orthogonal to $v_i,\,1\leq i\leq k$. Hence
\[
0\leq \hat{Q}(F_t),\quad \forall t\in\r.
\]

 This means that
\begin{equation}
\hat{Q}(\langle\hat{\Phi},a_j\rangle X_n,Y)^2\leq \hat{Q}(Y)\hat{Q}(\varphi_n\langle\hat{\Phi},a_j\rangle X_n),\quad \hat{Q}(Y)\geq0,\quad 1\leq j\leq 4.
\end{equation}
Making a computation like in the proof of the above theorem, we obtain that
\[
\hat{Q}(\varphi_n\langle\hat{\Phi},a_j\rangle X_n)=\int_{\hat{\Sigma}}|\hat{\nabla}\varphi_n|^2\langle\hat{\Phi},a_j\rangle^2|X_n|^2-\int_{\hat{\Sigma}}\varphi_n^2\langle \hat{L}(\langle\hat{\Phi},a_j\rangle X_n),\langle\hat{\Phi},a_j\rangle X_n\rangle,
\]
for all $n\in\n$ and $1\leq j\leq 4$.
Now, using Lemma~\ref{lm:Jacobi-harmonic} it is not difficult to see that
\begin{eqnarray*}
\langle \hat{L}(\langle\hat{\Phi},a_j\rangle X_n),\langle\hat{\Phi},a_j\rangle X_n\rangle=\langle\hat{\Phi},a_j\rangle\hat{\Delta}\langle\hat{\Phi},a_j\rangle|X_n|^2+\langle\hat{\Phi},a_j\rangle^2\langle \hat{L}X_n,X_n\rangle\\
+\frac{1}{2}\langle\hat{\nabla}\langle\hat{\Phi},a_j\rangle^2,\hat{\nabla}|X_n|^2\rangle=\frac{1}{2}\langle\hat{\nabla}\langle\hat{\Phi},a_j\rangle^2,\hat{\nabla}|X_n|^2\rangle
\end{eqnarray*}
But $\sum_{j=1}^4\langle\hat{\nabla}\langle\hat{\Phi},a_j\rangle^2,\hat{\nabla}|X_n|^2\rangle=0$, and so
\[
\sum_{j=1}^4\hat{Q}(\varphi_n\langle\hat{\Phi},a_j\rangle X_n)=\int_{\hat{\Sigma}}|\nabla\varphi_n|^2|X_n|^2.
\]
Using a similar argument like in the proof of Theorem~\ref{thm:indice-infinito}, $\lim \int_{\hat{\Sigma}}|\nabla\varphi_n|^2|X_n|^2=0$, and so from (4.3) we get
\[
0=\sum_{j=1}^4\lim\hat{Q}(\langle\hat{\Phi},a_j\rangle X_n,Y)^2=\sum_{j=1}^4\hat{Q}(\langle\hat{\Phi},a_j\rangle X,Y)^2,
\]
where $X=\lim X_n\in V$, after extracting a subsequence. As $Y$ is arbitrary, we finally get that  $\hat{L} (\langle\hat{\Phi},a\rangle X)=0,\,\forall a\in\r^4$. In particular and using Lemma~\ref{lm:Jacobi-harmonic}, we deduce that $0=\langle \hat{L}(\langle\hat{\Phi},a\rangle X),\hat{\Phi}\rangle=-2\langle X,a\rangle$, which is a contradiction.
\end{proof}
In \cite{TU} Torralbo and Urbano classified the compact stable minimal submanifolds of the product of two spheres. As a particular case of that classification,  the authors obtain that {\it the slices $\s^2\times\{p\},\,p\in\s^1(r)$ are the only compact stable minimal surfaces in $\s^2\times\s^1(r)$}. In the next result we study the index of a compact minimal surface of $\s^2\times\s^1(r)$.
\begin{theorem}\label{thm:indice-uno}
Let $\Phi:\Sigma\rightarrow\s^2\times\s^1(r),\,r\geq 1$, be a minimal immersion of a compact surface $\Sigma$.
\begin{enumerate}
\item If $\Sigma$ is orientable of genus $g$, then $\Index\,(\Phi)\geq \frac{2g-1}{5}$. Moreover $\Index\,(\Phi)=1$ if and only if $r=1$ and $\Phi$ is the totally geodesic embedding  $\s^1\times\s^1\subset\s^2\times\s^1$.
\item If $\Sigma$ is nonorientable, then $\Index\,(\Phi)\geq \frac{g-1}{5}$, where $g$ is the genus of the two-fold oriented covering of $\Sigma$.
\end{enumerate}
\end{theorem}
\begin{remark}
Note that when $r>1$, there are no compact orientable minimal surfaces of index one in $\s^2\times\s^1(r)$. Also, the idea used in the proof does not work for $r<1$. In this case, when $r<1$, not only the totally geodesic embedding $\s^1\times\s^1(r)\subset\s^2\times\s^1(r)$ has index one, but also any covering of $m$-sheets $\s^1\times\s^1(mr)\rightarrow\s^1\times\s^1(r)$ with $m\leq 1/r$.
\end{remark}
\begin{proof}
First we prove that $\Index\,(\Phi)\geq \frac{2g-1}{5}$. When $g=0$ the above inequality is irrelevant and when $g=1$ is known, because the only stable compact minimal surfaces of $\s^2\times\s^1(r)$ have genus zero. So we can assume that $g\geq 2$.
Let $m = \Index(\Phi)$. As in this case the immersion is two-sided, the Jacobi operator $L$ acts on functions. Let $\{\varphi_1,\ldots, \varphi_m\}$ be
 eigenfunctions of $L$ corresponding to the negative eigenvalue $\lambda_1,\dots,\lambda_m$. Considering  $\s^2 \times \s^1(r) \subset \r^5$ and $\{a_1, \ldots, a_5\}$ an orthonormal frame in $\r^5$,  we define the linear function $F: H(\Sigma) \rightarrow \r^{5m}$ given by
\[
F(X) = \left( \int_\Sigma \varphi_1 X,\dots,\int_\Sigma \varphi_m X\right),
\]
where $X$ is being considered as a $\r^5$-valuated function.

Suppose that $X \in \ker F$. Then $X$ is $L_2$-orthogonal to each $\varphi_i$, $1\leq i \leq m$ and hence $Q(X) \geq 0$. From Lemma~\ref{lm:Jacobi-harmonic}, we have that
\[
0 \leq Q(X) = -\int_\Sigma \prodesc{X}{\xi}^2 \left[ 2-(1+\frac{1}{r^2})|\xi^{\top}|^2\right] \leq 0,
\]
where the last inequality holds because we suppose $r^2 \geq 1$. This implies that either $r = 1$ and $|\xi^{\top}|^2=1$ or $\prodesc{X}{\xi} = 0$. In the first case, $\Sigma$ is a finite covering of the totally geodesic torus $\s^1 \times \s^1$, and so it has genus one, which is not the case. Hence $\prodesc{X}{\xi} = 0$, and from Lemma~\ref{lm:campos-armonicos}, $X = \lambda J\xi^{\top}$, with $\lambda \in \r$, that is, $\dim \ker F \leq 1$. As $2g = \dim H(\Sigma) = \dim \ker F + \dim \pIm F \leq 1 + 5m$, we get the result.

If $\Sigma$ is compact and nonorientable, let $\tilde{\Sigma}$ be the two-fold oriented covering of $\Sigma$ and $\tau$ the change of sheet in $\tilde{\Sigma}$. Using the same argument like in previous results, the index of $\Phi$ is the index of the Schrödinger operator $\tilde{L}=\tilde{\Delta}-\tilde{K}+1+|\tilde{\sigma}|^2/2$ acting on $C^{\infty}_-(\tilde{\Sigma})=\{f\in C^{\infty}(\tilde{\Sigma})\,|\,f\circ\tau=-f\}$. From Proposition 1, let $\{\varphi_1,\ldots, \varphi_m\}$ be eigenfunctions of $\tilde{L}$, with $\varphi_i\circ\tau=-\varphi_i$, corresponding to the negative eigenvalues $\lambda_1,\dots,\lambda_m$, where $m$ is the index of $\Phi$. In this case, we define the linear map  $F: H^-(\tilde{\Sigma}) \rightarrow \r^{5m}$ given by
\[
F(X) = \left( \int_{\tilde{\Sigma}} \varphi_1 X,\dots,\int_{\tilde{\Sigma}} \varphi_m X\right),
\]
where $X$ is being considered as a $\r^5$-valuated function. Following the same idea as above and taking into account that $\dim H^-(\tilde{\Sigma})=g$, we prove (2).

Finally, it is easy to check that the totally geodesic embedding $\s^1\times\s^1(r)\subset\s^2\times\s^1(r),\,r\geq 1$, has index one if and only if $r=1$.

Conversely, let $\Phi:\Sigma\rightarrow\s^2\times\s^1(r),\,r\geq 1$, be a minimal immersion of an orientable and compact surface $\Sigma$ with index one. If $\Phi=(\phi,\psi)$, for each $a\in\r^2$ the function $\langle\psi,a\rangle:\Sigma\rightarrow\r$ satisfies
\[
v(\langle\psi,a\rangle)=\langle\psi_*(v),a\rangle=\langle v,\xi\rangle\langle\xi,a\rangle.
\]
Hence
\[
\nabla\langle\psi,a\rangle=\langle\xi,a\rangle\,\xi^{\top}.
\]
Hence the Laplacian of $\langle\psi,a\rangle$ is given by
\[
\Delta\langle\psi,a\rangle=\langle\bar{\sigma}(\xi^{\top},\xi),a\rangle+\langle\xi,a\rangle\hbox{div}\,\xi^{\top}
=-\frac{1}{r^2}\langle \xi^{\top},\xi\rangle\langle\psi,a\rangle=-\frac{|\xi^{\top}|^2}{r^2}\langle\psi,a\rangle.
\]
So the Jacobi operator of the surface $\Sigma$ acting on $\langle\psi,a\rangle$ is given by
\[
L\,\langle\psi,a\rangle = \bigl((1 - 1/r^2)|\xi^{\top}|^2+|\sigma|^2\bigr)\langle\psi,a\rangle,
\]
and
\[
Q(\langle\psi,a\rangle)=-\int_{\Sigma}((1-1/r^2)|\xi^{\top}|^2+|\sigma|^2)\langle\psi,a\rangle^2\,dA\leq 0.
\]
If $V=\{\langle\psi,a\rangle,a\in\r^2\}$ and $\hbox{dim}\,V\leq 1$, then there exists a non-zero vector $a$ in $\r^2$ such that $\langle\psi,a\rangle=0$ and so  $\psi$ is constant. Hence $\Phi(\Sigma)$ is a slice of $\s^2\times\s^1(r)$, which is stable. So $\hbox{dim}\,V= 2$. Now, as the index of $\Sigma$ is one and $\xi^{\top}$ has only isolated zeroes, one gets that $\sigma=0$ and $r=1$. Then, $\Phi(\Sigma)$ is a finite covering of the totally geodesic surface $\s^1\times\s^1$. As the index is one, the surface must be $\s^1\times\s^1$ and the proof of (1) is complete.
\end{proof}

To finish, we will consider as ambient manifold $\r\p^2\times\r$. This $3$-manifold is nonorientable (its two-fold oriented covering is $\s^2\times\r$) and so the two-sidedness of its immersed surfaces is not related with their orientability. In any case, from Corollary 1, a {\it two-sided stable complete minimal surface of $\r\p^2\times\r$ must be compact and then the surface is a slice $\r\p^2\times\{t\}$}, which is stable. When the minimal surface $\Sigma$ is one-sided, then $\Sigma$ can be orientable or nonorientable. In the first case we classify the stable ones. Hence, only for nonorientable one-sided complete minimal surfaces of $\r\p^2\times\r$ the classification of the stable ones is still open.
\begin{theorem}\label{thm:estable}
Let $\Phi:\Sigma\rightarrow \r\p^2\times\r$ be a minimal immersion of an orientable complete surface. Then
$\Phi$ is stable if and only if
\begin{enumerate}
\item $\Sigma=\s^2, \Phi(\Sigma)=\r\p^2\times \{ t\},\,t\in\r$, or
\item $\Phi$ is the totally geodesic embedding of $\r\p^1\times\r$ into $\r\p^2\times\r$.
\end{enumerate}
\end{theorem}
\begin{remark}
Note that the stable totally geodesic embedding $\r\p^1\times\r\subset\r\p^2\times\r$ is the quotient, under the projection $\s^2\times\r\rightarrow\r\p^2\times\r$, of the totally geodesic embedding $\s^1\times\r\subset\s^2\times\r$ whose index is infinite.
\end{remark}
\begin{proof}
It is clear that the totally geodesic immersion of $\s^2$ into $\r\p^2\times\r$ given in (1) is stable. Also, as the totally geodesic embedding $\r\p^1\subset\r\p^2$ is a stable geodesic, it is not difficult to check that the embedding given in (2) is also stable.

Conversely, we suppose that $\Phi$ is stable. If $\Phi$ is two-sided, then Corollary 1 says that $\Sigma$ is compact. As $\Sigma$ is orientable, we obtain that $\Sigma=\s^2$ and $\Phi(\Sigma)=\r\p^2\times\{t\}$.

If $\Phi$ is one-sided then $\Sigma$ is non-compact, because the only two compact examples are two-sided.

\[
\xymatrix{
(\tilde{\Sigma}, \tau)   \ar[d]_{2:1} \ar[rr]^{\tilde{\Phi}=(\Psi,h)}    &   & \s^2\times\r \ar[d] \\
\Sigma  \ar[rr]^{\Phi}_{1\text{-sided}}                       &   & \r\p^2\times\r
}
\]
Let $\s^2\times\r\rightarrow\r\p^2\times\r$ be the projection. As $\Sigma$ is orientable, $\Phi$ does not lift to $\s^2\times\r$, and so the above covering induces a two-fold covering $\tilde{\Sigma}\rightarrow\Sigma$ of a connected surface $\tilde{\Sigma}$. If $\tau$ is the change of sheet and $\tilde{\Phi}=(\Psi,h):\tilde{\Sigma}\rightarrow\s^2\times\r$ the corresponding minimal immersion, then $\Psi\circ\tau=-\Psi$ and $\tilde{\Phi}$ is two-sided.

Now we can identify sections of the normal bundle of $\Phi$ with functions on $\tilde{\Sigma}$ which are odd with respect to $\tau$ in the following way:
\begin{eqnarray*}
\Gamma(T^{\perp}\Sigma)&\equiv& C^{\infty}_-(\tilde{\Sigma})=\{f\in C^{\infty}(\tilde{\Sigma})\,|\,f\circ\tau=-f\} \\
\eta&\equiv& f,\quad \hbox{if}\,\,\tilde{\eta}=fN,
\end{eqnarray*}
where $\tilde{\eta}$ is the lift of $\eta$ to $\tilde{\Phi}$ and $N$ is a unit normal vector field to $\tilde{\Phi}$. As $\Phi$ is one-sided, then $N\circ\tau=-N$ and so, $f\circ\tau=-f$, because $\tilde{\eta}\circ\tau=\tilde{\eta}$.

Moreover, it is easy to check that if  $\eta\in\Gamma_0(T^{\bot}\Sigma)$ and  $f\in C^{\infty}_{-}(\tilde{\Sigma})$ is the corresponding function with $ \tilde{\eta}=fN$, then
\[
2Q(\eta)=\tilde{Q}(\tilde{\eta})=\tilde{Q}(f)=-\int_{\tilde{\Sigma}}f\tilde{L}f,
\]
where
\begin{eqnarray*}
\tilde{L}:C^{\infty}_-(\tilde{\Sigma})\rightarrow C^{\infty}_-(\tilde{\Sigma})\\
\tilde{L}=\tilde{\Delta}-\tilde{K}+(1+|\tilde{\sigma}|^2/2).
\end{eqnarray*}
Given $\varphi\in C_0^{\infty}(\Sigma)$, its lift $\tilde{\varphi}\in C_0^{\infty}(\tilde{\Sigma})$ satisfies $\tilde{\varphi}\circ\tau=\tilde{\varphi}$. Then, the $\r^3$-valued function $\tilde{\varphi}\Psi$ ($\Psi:\tilde{\Sigma}\rightarrow \s^2,\,\Psi\circ\tau=-\Psi)$ has compact support and $\tilde{\varphi}\Psi\circ\tau=-\tilde{\varphi}\Psi$.
As $\Phi$ is stable,
\begin{equation}
0\leq\tilde{Q}(\tilde{\varphi}\Psi).
\end{equation}
Now, we compute $\tilde{Q}(\tilde{\varphi}\Psi)$. First, $\tilde{L}(\tilde{\varphi}\Psi)=(\tilde{\Delta}\tilde{\varphi})\Psi+\tilde{\varphi}\tilde{L}\Psi+2\nabla^E_{\tilde{\nabla}\tilde{\varphi}}\Psi$, where $\nabla^E$ is the connection on $\r^4$.

First, $\tilde{\Delta}\Psi=- (\sum_{i=1}^2|\Psi_*e_i|^2)\Psi=-(1+\tilde{C}^2)\Psi$, where $\{e_1,e_2\}$ is an orthonormal reference on $T\tilde{\Sigma}$. Hence, using the Gauss equation of $\tilde{\Phi}$, we obtain $\tilde{L}(\Psi)=-2\tilde{C}^2\Psi+|\tilde{\sigma}|^2\Psi=-2\tilde{K}\Psi$. Hence,
\begin{eqnarray*}
\tilde{Q}(\tilde{\varphi}\Psi)&=&-\int_{\tilde{\Sigma}}\tilde{\varphi}\tilde{\Delta}\tilde{\varphi}d\tilde{A}+\int_{\tilde{\Sigma}}2\tilde{\varphi}^2\tilde{K}d\tilde{A}
-\int_{\tilde{\Sigma}}\langle\nabla^E_{\tilde{\nabla}\tilde{\varphi}^2}\Psi,\Psi\rangle d\tilde{A}\\
&=&\int_{\tilde{\Sigma}}(|\tilde{\nabla}\tilde{\varphi}|^2+2\tilde{\varphi}^2\tilde{K})d\tilde{A}.
\end{eqnarray*}
As $\tilde{K}$ and $\tilde{\varphi}$ are even with respect to $\tau$, (4.4) joint with the above computation imply that
\[
0\leq\tilde{Q}(\tilde{\varphi}\Psi)=\frac{1}{2}\int_{\Sigma}(|\nabla\varphi|^2+2K\varphi^2)dA\leq\int_{\Sigma}(|\nabla\varphi|^2+K\varphi^2)dA,
\]
for any $\varphi\in C_0^{\infty}(\Sigma)$.

The above inequality means that the  Schr\"odinger operator $\Delta-K$, on the complete and orientable Riemannian surface $\Sigma$, satisfies $\Index(\Delta-K)=0$.

From Theorem 1, and as $\Sigma$ is non-compact, we have that either $\Sigma$ is conformally equivalent to $\c$, which is impossible because $\Sigma$ admits a connected $2$-fold covering, or $\Sigma$ is a flat cylinder. In this case (see Section 3), it is not difficult to check that $\Delta |\xi^{\top}|^2=4(1-|\xi^{\top}|^2)^2$. So $|\xi^{\top}|^2$ is a subharmonic function that satisfies $|\xi^{\top}|^2\leq 1$. As $\Sigma$ is complete and flat, the maximum principle implies that $|\xi^{\top}|^2$ is constant, and so $|\xi^{\top}|^2=1$. In particular $\Phi$ is a totally geodesic immersion. Hence $\Sigma$ must be a finite covering of the totally geodesic embedding of $\r\p^1\times\r$ into $\r\p^2\times\r$, but, among then, only $\r\p^1\times\r$ is stable.
\end{proof}
The proof given in Theorem~\ref{thm:estable}, with minor changes, allows to prove the last result in the paper:
\begin{theorem}\label{thm:compact-estable}
Let $\Phi:\Sigma\rightarrow \r\p^2\times\s^1(r)$ be a minimal immersion of an orientable complete surface. Then
$\Phi$ is stable if and only if
\begin{enumerate}
\item $\Sigma=\s^2, \Phi(\Sigma)=\r\p^2\times \{ p\},\,p\in\s^1(r)$, or
\item $\Phi$ is the totally geodesic immersion of $\r\p^1\times\r$ into $\r\p^2\times\s^1(r)$, or
\item $\Phi$ is the totally geodesic embedding of $\r\p^1\times\s^1(r)$ into $\r\p^2\times\s^1(r)$.
\end{enumerate}
\end{theorem}

\end{document}